\documentclass[preprint,12pt]{elsarticle}

\usepackage{amssymb}
\usepackage{soul, color}
\usepackage{graphicx}          
\usepackage{amsthm}                               
 \usepackage{epsfig}

\usepackage{lipsum}
\usepackage{amsmath}
\usepackage{cleveref}
\newtheorem{lemma}{Lemma}
\newtheorem{theorem}{Theorem}   
\newtheorem{ass}{Assumption} 
\newtheorem{remark}{Remark}
\makeatletter
\def\ps@pprintTitle{%
 \let\@oddhead\@empty
 \let\@evenhead\@empty
 \def\@oddfoot{}%
 \let\@evenfoot\@oddfoot}
\makeatother

\date{}
\begin{document}

\begin{frontmatter}



\title{Seeking Nash Equilibrium in Non-Cooperative Differential Games}


\author[label1]{Zahra Zahedi }
\author[label1]{Alireza Khayatian}
\author[label1]{Mohammad Mehdi Arefi\corref{cor1}}
\ead{arefi@shirazu.ac.ir}
\cortext[cor1]{Corresponding author}
\author[label2]{Shen Yin}
\fntext[]{Journal of Vibration and Control\\
https://doi.org/10.1177/10775463221122120}

\address[label1]{The Department of Power and Control Engineering, School of Electrical and Computer Engineering, Shiraz University, Shiraz 71946-84636, Iran}
\address[label2]{Department of Mechanical and Industrial Engineering, Faculty of Engineering, Norwegian University of Science and Technology , Trondheim 7033, Norway}

\begin{abstract}
This paper aims at investigating the problem of fast convergence to the Nash equilibrium (NE) for N-Player noncooperative differential games. The proposed method is such that the players attain their NE point without steady-state oscillation (SSO) by measuring only their payoff values with no information about payoff functions, the model and also the actions of other players are not required for the players. The proposed method is based on an extremum seeking (ES) method, and moreover, compared to the traditional ES approaches, in the presented algorithm, the players can accomplish their NE faster. In fact, in our method the amplitude of the sinusoidal excitation signal in classical ES is adaptively updated and exponentially converges to zero. In addition, the analysis of convergence to NE is provided in this paper. Finally, a simulation example confirms the effectiveness of the proposed method. 
\end{abstract}




\end{frontmatter}


\section{Introduction}
%
%
%
%
The problem of finding an algorithm to attain NE has inspired many researchers thanks to the vast variety of applications of differential non-cooperative games in areas such as motion planning \cite{c1}, \cite{c2}, \cite{c3}, formation control \cite{c4}, \cite{c5}, wireless networks \cite{c6}, \cite{c7}, mobile sensor networks \cite{c8}, \cite{c9}, network security \cite{c10} and demand side management in smart grid \cite{c11}, \cite{c12}.\\
\indent The majority of algorithms which are designed to achieve convergence to NE are based on the model information and observation of other player's actions. These algorithms usually are designed on the basis of best response and fictitious play strategy. In \cite{c13}, each agent plays a best response strategy in a non-myopic Cournot competition. A form of dynamic fictitious and gradient play strategy in a continuous time form of repeated matrix games have been introduced in \cite{c14} and the convergence to NE has been shown. Distributed iterative algorithms have been considered in \cite{c15} to compute the equilibria in general class of non-quadratic convex games. The authors in \cite{c16} have presented two distributed learning algorithms in which players remember their own payoff values and actions from the last play; also the convergence to the set of Nash equilibria was proved. An algorithm based on the combination of the support enumeration method and the local search method for finding NE has been designed in \cite{c17}.\\
\indent Furthermore, many works are uncoupled which means each player generates its actions based on its own payoff not the actions or payoffs of the opponents. A new type of learning rule which is called regret testing in a finite game was introduced in \cite{c18} in which behaviors of players converge to the set of Nash equilibria. In \cite{c19}, almost certain convergence to NE in a finite game with uncoupled strategy, and possibility and impossibility results are studied.\\ 
\indent Moreover, some of the non-model based algorithms are a kind of extremum seeking control algorithm (ESC) which is a real time optimization tool that can find an extremum value of an unknown mapping \cite{c20}, \cite{c21}, \cite{c22}. For example, in \cite{c23}, the Nash seeking problem for \textit{N}-player non-cooperative games with static quadratic and dynamic payoff functions is presented, which is based on the sinusoidal perturbation extremum seeking approach. The same analysis for dynamic systems with non-quadratic payoffs has been considered in \cite{c24}, \cite{c25}. In \cite{oliveira2021nash}, they proposed a strategy for locally stable convergence to NE in quadratic non-cooperative games which are subject to diffusion PDE dynamic. The problem of attaining NE in non-cooperative games based on the stochastic extremum seeking is studied in \cite{c26}. Also, learning of a Generalized NE in strongly monotone games with nonlinear dynamical agents has been studied in \cite{krilavsevic2021learning}. However, usually the convergence to NE in all of the non-model based extremum seeking methods mentioned have steady-state oscillation.\\
\indent In this paper, we have modified an extremum seeking control without steady-state oscillation (ESCWSSO) algorithm \cite{c27} to design a seeking scheme to solve this problem for an \textit{N}-player differential non-cooperative game, which is faster than conventional extremum seeking algorithms in achieving NE. In this algorithm, the players can generate their actions only by measuring their own payoff value with no need for information regarding the model, details of the payoff function, and actions of other players. More importantly, the NE can be achieved fast and without steady-state oscillation, because the amplitude of excitation sinusoidal signal in classical extremum seeking is adjusted to exponentially converge to zero. As a result, the convergence will be fast and the improper effects of steady-state oscillation will be eliminated. The similar analyses have been provided in \cite{c28} and \cite{c29} for static non-cooperative games with quadratic and non-quadratic payoff functions. However, since they have investigated static games, dynamic systems are excluded from their analysis. Additionally, authors in \cite{c30} have proposed a new algorithm for fast and without steady state oscillation convergence to NE with a local dynamic within the algorithm for static non-cooperative games.\\ 
\indent The main contributions of the paper are as follows
\begin{itemize}
\item Comparing to the previous works for achieving NE, this paper can achieve the NE fast and without steady state oscillation.
\item In this paper we modified the previous algorithm \cite{c27} in order to be applicable on differential games which are multi-agent systems and proved the convergence of NE in such games.
\item In comparison to the previous works in \cite{c29}, and \cite{c30} which all were on static games, this paper studied differential and dynamic games which can make this algorithm applicable for a vast variety of new applications.
\end{itemize}
\indent The rest of this paper is organized as follows. In \cref{sec2}, the Nash seeking algorithm and the general description of the problem are stated. \cref{sec3} includes convergence and stability analysis. A numerical example with simulation is presented in \cref{sec4}. Finally, a conclusion is provided in \cref{sec5}.

\section{Preliminaries}
\label{sec2}

\indent In this section, we consider the problem of fast convergence to NE without steady state oscillation in non-cooperative differential games with \textit{N} players in which players can converge to their NE fast and without oscillation only with the measurement of their own payoff values, which means that they do not need to have any knowledge of the model and actions of other players.\\
\indent Consider the following nonlinear model of player $i$ in an \textit{N}-player non-cooperative differential game:
\begin{equation}
\label{eq1}
\dot{x}=f(x,u)
\end{equation}
\begin{equation}
\label{eq2}
J_i = j_i(x)
\end{equation}
where $i\in\{1, \hdots, N\}$, $u\in R^N$ is actions of players which is $u=[u_1, \hdots, u_N]$, $x$ is the state and $J_i$, $j_i$ and $f$ are smooth functions where $j_i:R^n\to R$, $f:R^n\times R^N\to R^n$ and $J_i\in R$. At first, let's make the following assumptions about the game, which are the same as \cite{c24}.
\begin{ass}
\label{as1}
There exists a smooth function $l:R^N\to R^n$ such that $f(x,u)=0$ if and only if $x=l(u)$.
\end{ass}
\begin{ass}
\label{as2}
The equilibrium of system (\ref{eq1}) which is $x=l(u)$ is locally exponentially stable for all $u\in R^N$.
\end{ass}
Thus, these assumptions mean that any actions of the players can be designed to stabilize the equilibrium without any knowledge about the model, other players' actions or form of $j_i$, $f$ and $l$.\\
Each player $i$ employs the following algorithm in order to generate its actions to achieve NE.
\begin{equation}
\label{eq3}
\begin{split}
\dot{\hat{u}}_i &= k_i(J_i-n_i)\sin(\omega_it+\phi_i), \\
 u_i &= \hat{u}_i+a_i\sin(\omega_it+\phi_i),\\
 \dot{a}_i &= -\omega_{li}a_i+b_i\omega_{li}(J_i-n_i),\\ 
 \dot{n}_i &= -\omega_{hi}n_i+\omega_{hi}J_i, 
\end{split}
\end{equation}
where $k_i$ and $\omega_i$ are positive constants, $\omega_i=\bar{\omega}\tilde{\omega}_i$ in which $\bar{\omega}$ is a positive rational number and $\tilde{\omega}_i$ is a positive real number, $J_i$ is the measurement value of payoff, $n_i$ is the low-frequency components of $J_i$, $b_i$ and $\phi_i$ are constants where $b_i$ can adjust the speed of convergence, $\omega_{li}$ and $\omega_{hi}$ are cut-off frequencies for low pass and high pass filters respectively. Also, $\omega_{li}$ and $b_i$ are new parameters which can be used instead of excitation sinusoidal signal's amplitude in classical extremum seeking control, which are designed to make $a_i$ positive. \Cref{fig1} shows the diagram of the algorithm in an \textit{N}-player non-cooperative differential game.\\
\begin{figure}[h]
\begin{center}
\includegraphics[height=8cm]{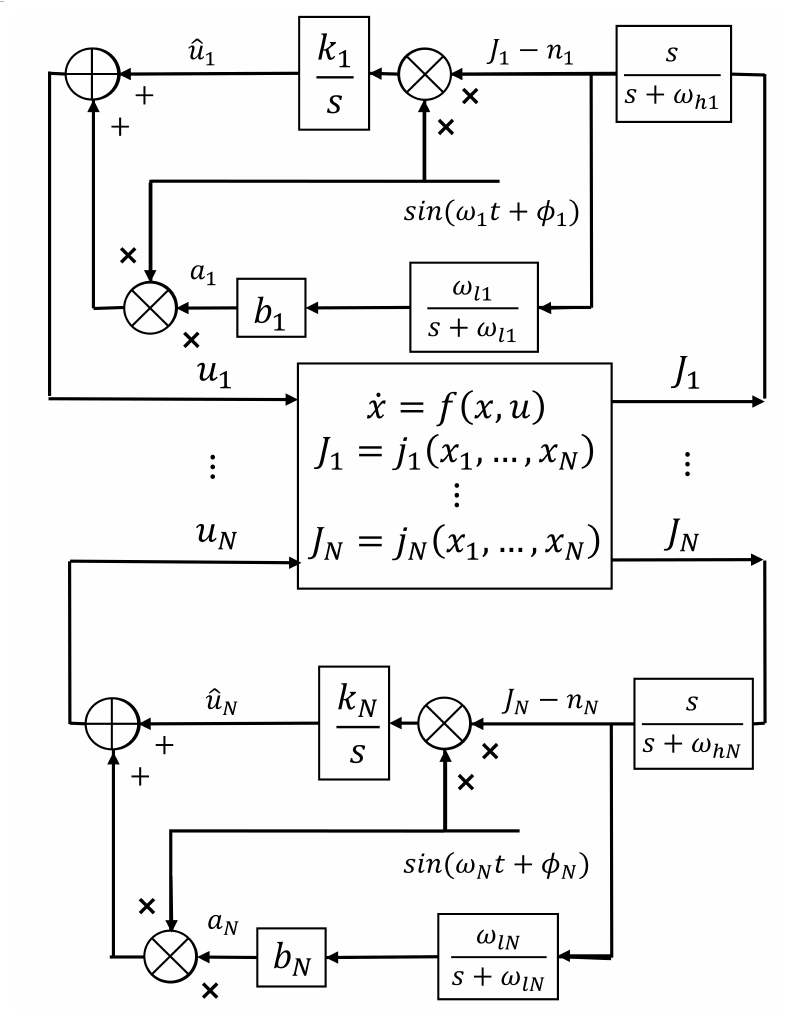}    
\caption{NE seeking without steady-state oscillation scheme for an \textit{N}-player non-cooperative differential game}  
\label{fig1}                                 
\end{center}                                 
\end{figure}
\indent We introduce the following errors for more analysis:
\begin{equation}
\label{eq4}
\begin{split}
\tilde{u}_i &= \hat{u}_i-u_i^\ast,\\
\tilde{n}_i &= n_i-j_i\circ l(u^\ast). 
\end{split}
\end{equation}
where $u^\ast$ is a vector of players' NE $u^\ast=[u_1^\ast, \hdots, u_N^\ast]$.\\
\indent By substituting (\ref{eq4}) into (\ref{eq3}), the following equation is obtained:
\begin{equation}
\label{eq5}
\begin{split}
\dot{x}&=f(x,\tilde{u}+u^\ast+a\times \eta(t))\\
\dot{\tilde{u}}_i &= k_i(j_i(x)-j_i\circ l(u^\ast)-\tilde{n}_i)\sin(\omega_it+\phi_i), \\
 \dot{a}_i &= -\omega_{li}a_i(t)+b_i\omega_{li}(j_i(x)-j_i\circ l(u^\ast)-\tilde{n}_i),\\ 
 \dot{\tilde{n}}_i &= -\omega_{hi}\tilde{n}_i(t)+\omega_{hi}(j_i(x)-j_i\circ l(u^\ast)), 
\end{split}
\end{equation}
where $\eta_i(t)=\sin(\omega_it+\phi_i)$, $\tilde{u}=[\tilde{u}_1, \hdots, \tilde{u}_N]$, $a=[a_1(t), \hdots, a_N(t)]$, and $\eta(t)=[\eta_1(t), \hdots, \eta_N(t)]$. In addition, in the equation $a\times \eta(t)$, $(\times)$ sign  means the entry-wise product of two vectors. \\
\indent For more analysis, the design parameters are chosen as follows:
\begin{equation}
\label{eq6}
\begin{split}
\omega_{li} &=\omega\omega_{Li}=\omega\epsilon\omega_{Li}^\prime=O(\omega\epsilon\delta),\\
\omega_{hi} &=\omega\omega_{Hi}=\omega\delta\omega_{Hi}^\prime=O(\omega\delta),\\
k_{i} &=\omega K_{i}=\omega\delta K_{i}^\prime=O(\omega\delta),
\end{split}
\end{equation}
where $\epsilon$ and $\delta$ are small constant parameters, and $\omega_{Hi}^\prime$, $K_i^\prime$ and $\omega_{Li}^\prime$ are $O(1)$ positive constants. Afterwards, the system in time scale $\tau=\bar{\omega}t$ is obtained.
\begin{equation}
\label{eq7}
\bar{\omega}\frac{dx}{d\tau}=f(x,\tilde{u}+u^\ast+a\times\eta(\tau))
\end{equation}
\begin{equation}
\label{eq8}
\begin{split}
\frac{d}{d\tau}\left[\begin{array}{c} \tilde{u}_i(\tau)\\ a_i(\tau)\\\tilde{n}_i(\tau)\end{array} \right]=\delta\left[\begin{array}{l} 
K_i^\prime(j_i(x)-j_i\circ l(u^\ast)-\tilde{n}_i)\eta_i(\tau)\\
\epsilon\omega_{Li}^\prime (b_i(j_i(x)-j_i\circ l(u^\ast)-\tilde{n}_i)-a_i)\\
\omega_{Hi}^\prime(j_i(x)-j_i\circ l(u^\ast)-\tilde{n}_i)\end{array}\right]
\end{split}
\end{equation}
Furthermore, the following assumptions are made to ensure the existence of NE $u^\ast$:
\begin{ass}
\label{as3}
The game admits at least one stable NE $u^\ast$ in which 
\begin{equation}
\label{eq9}
\frac{\partial j_i\circ l}{\partial u_i}(u^\ast)=0, \hspace{10pt} \frac{\partial^2 j_i\circ l}{\partial u_i^2}(u^\ast)<0.
\end{equation}
\end{ass}
\begin{ass}
\label{as4}
The following matrix is diagonally dominant and consequently nonsingular. Hence, according to Assumption \ref{as3} and Gershgorin Circle theorem \footnote{According to Gershgorin Circle theorem, a diagonally dominant matrix is stable (Hurwitz) when all elements of the diagonal have negative real part.} \cite{c31}, it is also stable Hurwitz.
\begin{equation}
\label{eq10}
\Delta=\left[\begin{array}{cccc} \frac{\partial^2 j_1ol(u^\ast)}{\partial u_1^2}&\frac{\partial^2 j_1ol(u^\ast)}{\partial u_1\partial u_2}&\ldots&\frac{\partial^2 j_1ol(u^\ast)}{\partial u_1\partial u_N}\\\frac{\partial^2 j_2ol(u^\ast)}{\partial u_1\partial u_2}&\frac{\partial^2 j_2ol(u^\ast)}{\partial u_2^2}&&\frac{\partial^2 j_2ol(u^\ast)}{\partial u_2\partial u_N}\\\vdots&&\ddots&\vdots\\\frac{\partial^2 j_Nol(u^\ast)}{\partial u_1\partial u_N}&&&\frac{\partial^2 j_Nol(u^\ast)}{\partial u_N^2}\end{array}\right]
\end{equation}
\end{ass}
\section{Main Results}
\label{sec3}
In this section, local stability and convergence analysis of proposed algorithm will be studied.
\subsection{Averaging analysis}
For averaging analysis, let's freeze $x$ in quasi-steady state equilibrium:
\begin{equation}
\label{eq11}
x=l(\tilde{u}+u^\ast+a\times\eta(\tau))
\end{equation}
By substituting (\ref{eq11}) to (\ref{eq8}), the reduced system is given as follows:
\begin{equation}
\label{eq12}
\begin{split}
&\frac{d}{d\tau}\left[\begin{array}{c} \tilde{u}_{ri}(\tau)\\ a_{ri}(\tau)\\\tilde{n}_{ri}(\tau)\end{array} \right]=\\
&\delta{\small \left[\begin{array}{l} 
K_i^\prime(j_i\circ l(\tilde{u}_r+u^\ast+a_r\times\eta)-j_i\circ l(u^\ast)-\tilde{n}_{ri})\eta_i(\tau)\\
\epsilon\omega_{Li}^\prime (b_i(j_i\circ l(\tilde{u}_r+u^\ast+a_r\times\eta)-j_i\circ l(u^\ast)-\tilde{n}_{ri})-a_{ri})\\
\omega_{Hi}^\prime(j_i\circ l(\tilde{u}_r+u^\ast+a_r\times\eta)-j_i\circ l(u^\ast)-\tilde{n}_{ri})\end{array}\right]}
\end{split}
\end{equation}
in which if we consider $j_i\circ l(\tilde{u}_r+u^\ast+a_r\times\eta)-j_i\circ l(u^\ast)=h_i\circ l(\tilde{u}_r^{av}+a_r^{av}\times\eta)$, regarding to Assumption \ref{as3}, we have
\begin{equation}
\label{eq13}
h_i\circ l(0)=0, \hspace{10pt} \frac{\partial h_i\circ l}{\partial u_i}(0)=0, \hspace{10pt} \frac{\partial^2 h_i\circ l}{\partial u_i^2}(0)<0.
\end{equation}
Since (\ref{eq12}) is in the proper form of averaging theory \cite{c32}, the averaging system is obtained as follows:
\begin{equation}
\label{eq14}
\begin{split}
&\dot{\tilde{u}}_{ri}^{av}(\tau)=\delta({\lim_{T\to+\infty}}\frac{K_i^\prime}{T}\int_0^T(h_i\circ l(\tilde{u}_r^{av}+a_r^{av}\times\eta)\eta_i(\tau)\,d\tau),\\
&\dot{a}^{av}_{ri}(\tau)=\delta\epsilon\omega_{Li}^\prime (b_i({\lim_{T\to+\infty}}\frac{1}{T}\int_0^T h_i\circ l(\tilde{u}_r^{av}+a_r^{av}\times\eta)\,d\tau\\&-\tilde{n}^{av}_{ri})-a^{av}_{ri}),\\
&\dot{\tilde{n}}^{av}_{ri}(\tau)=\delta\omega_{Hi}^\prime({\lim_{T\to+\infty}}\frac{1}{T}\int_0^T h_i\circ l(\tilde{u}_r^{av}+a_r^{av}\times\eta)\,d\tau-\\&\tilde{n}^{av}_{ri}).
\end{split}
\end{equation}
Therefore, the following result can be derived.
\begin{theorem}
\label{th1}
Consider system (\ref{eq12}) for an \textit{N}-player game, regarding to Assumptions \ref{as3} and \ref{as4} where $\omega_i\not=\omega_j$, $\omega_i\not=\omega_j+\omega_k$, $2\omega_i\not=\omega_j+\omega_k$, $\omega_i\not=2\omega_j+\omega_k$, $\omega_i\not=2\omega_j$, $\omega_i\not=3\omega_j$ and $\frac{\omega_i}{\omega_j}$ is rational, for all $i, j, k\in\{1,\hdots,N\}$. Then, with constants $\sigma$, $\overline{\epsilon}$, $\overline{\delta}>0$ which $0<\epsilon<\overline{\epsilon}$ and $0<\delta<\overline{\delta}$, there exists a neighborhood for equilibrium point of average system $(\tilde{u}_r^e, 0, \tilde{n}_r^e)$ that $(\tilde{u}_r, a_r, \tilde{n}_r)$ will exponentially converge to that neighborhood.
\end{theorem}
\begin{proof}
According to this fact that $\frac{d\epsilon}{d\tau}=0$ and by using the center manifold technique \cite{c33}, we rewrite (\ref{eq14}) in the following form:
\begin{equation}
\label{eq15}
\dot{z}=\frac{d}{d\tau}\left[\begin{array}{c} \tilde{u}_{ri}^{av}(\tau)\\a^{av}_{ri}(\tau)\\ \epsilon\end{array}\right]=A_1z+g_1(z,y)
\end{equation}
\begin{equation}
\label{eq16}
\dot{y}=\frac{d\tilde{n}^{av}_{ri}(\tau)}{d\tau}=A_2z+g_2(z,y)
\end{equation}
where $A_1=\left[\begin{array}{ccc}0&0&0\\0&0&0\\0&0&0\end{array}\right]$,$g_1(z,y)=\delta\left[\begin{array}{c}g_{11}\\g_{12}\\0\end{array}\right]$, $g_{11}=\lim_{T\to+\infty}\frac{K_i^\prime}{T}\int_0^T(h_i\circ l(\tilde{u}_r^{av}+a_r^{av}\times\eta))\eta_i\,d\tau)$ and $g_{12}=\epsilon\omega_{Li}^\prime (-a^{av}_{ri}+b_i(-y+{\lim_{T\to+\infty}}\frac{1}{T}\int_0^T h_i\circ l(\tilde{u}_r^{av}+a_r^{av}\times\eta)\,d\tau))$, $g_2(z,y)=\lim_{T\to+\infty}\frac{\delta\omega_{Hi}^\prime}{T}\int_0^T h_i\circ l(\tilde{u}_r^{av}+a_r^{av}\times\eta)\,d\tau)$ and $A_2=-\delta\omega_{Hi}^\prime$.\\
\indent Firstly, consider the following lemmas about center manifold.
\begin{lemma}
\label{lem1}
There exists a center manifold $y=q(z)$, $\Vert z\Vert<0$ for (\ref{eq15}) and (\ref{eq16}), in which $q\in C^2$ and $\sigma >0$, if we have the following conditions:
\begin{enumerate}
    \item $A_1$ and $A_2$ are constant matrices where all the eigenvalues of $A_1$ and $A_2$ have zero real part and negative real parts respectively.
    \item $g_1$ and $g_2$ are $C^2$ so that $g_i(0,0)=0$, $\frac{\partial g_i}{\partial z}(0,0)=0$ and $\frac{\partial g_i}{\partial y}(0,0)=0, $ $\forall$ $i=\{1, 2\}$. \hfill \cite{c33}
\end{enumerate}
\end{lemma}
\begin{lemma}
\label{lem2}
Consider a $\Psi(z)\in C^1$ where $\Psi(0)=0$ and $\frac{\partial\Psi}{\partial z}(0)=0$, such that $M(\Psi(z))=O(\Vert z\Vert^m)$ where $m>1$, then if $\Vert x\Vert\to 0$ we have $\vert q(z)-\psi(z)\vert=O(\Vert z\Vert^m)$ \cite{c33}.
\end{lemma}
\begin{lemma}
\label{lem3}
The stability situation of the origin of (\ref{eq15}) and (\ref{eq16}) is the same as the origin of the following equation. \cite{c33} 
\begin{equation}
    \label{177}
    \dot{z}=A_1z+g_1(z,q(z))
\end{equation}
which means that if the origin of (\ref{177}) is either stable, exponentially stable, or unstable then it would be the same for the origin of (\ref{eq15}) and (\ref{eq16}).
\end{lemma}
Therefore, according to Lemma \ref{lem1}, a center manifold $y=q(z)$ can be approximated, because the assumptions of Lemma \ref{lem1} are held in (\ref{eq15}) and (\ref{eq16}). Thus, we have the following equation: 
\begin{equation}
\label{eq17}
\begin{split}
&M(\Psi(z))=\frac{\partial\Psi}{\partial\tilde{u}^{av}_{ri}}\frac{\partial\tilde{u}^{av}_{ri}}{\partial\tau}+\frac{\partial\Psi}{\partial a^{av}_{ri}}\frac{\partial a^{av}_{ri}}{\partial\tau}+\frac{\partial\Psi}{\partial\epsilon}\frac{\partial\epsilon}{\partial\tau}\\&+\delta\omega_{Hi}^\prime(\Psi(z)-\lim_{T\to+\infty}\frac{1}{T}\int_0^T h_i\circ l(\tilde{u}_r^{av}+a_r^{av}\times\eta)\,d\tau).
\end{split}
\end{equation}
By considering Assumption \ref{as4}, if $\Psi(z)=\frac{\partial^2h_i\circ l(0)}{2\partial u^2_i}\tilde{u}_{ri}^{av^2}+\frac{\partial^2h_i\circ l(0)}{4\partial u^2_{ri}}a^{av^2}_{ri}$, then $M(\Psi(z))$ is $O(\vert\tilde{u}^{av}_{ri}\vert^3+\vert a^{av}_{ri}\vert^3+\vert\epsilon\vert^3)$. As a result, by Lemma \ref{lem2}, the center manifold is obtained as follows:
\begin{equation}
\label{eq18}
\begin{split}
y=q(z)=\Psi(z)=&\frac{\partial^2h_i\circ l(0)}{2\partial u^2_{ri}}\tilde{u}_{ri}^{av^2}+\frac{\partial^2h_i\circ l(0)}{4\partial u^2_{ri}}a^{av^2}_{ri}\\&+O(\vert\tilde{u}^{av}_{ri}\vert^3+\vert a^{av}_{ri}\vert^3+\vert\epsilon\vert^3).
\end {split}
\end{equation}
By substituting center manifold $y$ into (\ref{eq14}), the following equations are given:
\begin{equation}
\label{eq19}
\begin{split}
&\dot{\tilde{u}}_{ri}^{av}(\tau)=\delta({\lim_{T\to+\infty}}\frac{K_i^\prime}{T}\int_0^T(h_i\circ l(\tilde{u}_r^{av}+a_r^{av}\times\eta))\eta_i(\tau)\,d\tau)
\end{split}
\end{equation}
\begin{equation}
\label{eq20}
\dot{a}^{av}_{ri}(\tau)=-\delta\epsilon\omega_{Li}^\prime a^{av}_{ri}+b_i\delta\epsilon\omega_{Li}^\prime O(\vert\tilde{u}^{av}_{ri}\vert^3+\vert a^{av}_{ri}\vert^3+\vert\epsilon\vert^3).
\end{equation}
According to Lemma \ref{as3}, we can use (\ref{eq19}) and (\ref{eq20}) in order to analyze the stability of (\ref{eq14}).\\
The equilibrium of (\ref{eq19}) which is $\tilde{u}_r^e=[\tilde{u}_{r1}^e, \hdots, \tilde{u}_{rN}^e]$ admits the following equation:
\begin{equation}
\label{eq21}
\begin{split}
0&={\lim_{T\to+\infty}}\frac{1}{T}\int_0^T h_i\circ l(\tilde{u}_r^{av}+a_r^{av}\times\eta)\eta_i(\tau)\,d\tau.
\end{split}
\end{equation}
We assume that the equilibrium point $\tilde{u}_{ri}^e$ for all $i\in\{1, \hdots, N\}$ is as follows:
\begin{equation}
\label{eq22}
\tilde{u}^e_{ri}=\sum_{j=1}^Nc_j^ia_{rj}^{av}+\sum_{j=1}^N\sum_{k\ge j}^Nd_{jk}^ia_{rj}^{av}a_{rk}^{av}+O(\max_ia_{ri}^{av^3}).
\end{equation}
\indent In regard to (\ref{eq13}), if the Taylor polynomial approximation \cite{c23}, \cite{c34} of $h_i\circ l$ about zero is substituted in (\ref{eq21}) (the details of Taylor polynomial approximation and the integrals which are used are given in the Appendix), then the following equation is acquired by applying the averaging technique:
\begin{equation}
\label{eq23}
\begin{split}
&0=\frac{a_{ri}^{av}}{2}\left[\sum_{j\not=i}^N\tilde{u}_{rj}^e\frac{\partial^2h_i\circ l}{\partial u_{ri}\partial u_{rj}}(0)+\tilde{u}_{ri}^e\frac{\partial^2h_i\circ l}{\partial u_{ri}^2}(0)\right.\\&+\frac{\tilde{u}_{ri}^{e^2}}{2}\frac{\partial^3h_i\circ l}{\partial u_{ri}^3}(0)+\sum_{j\not=i}^N(\frac{\tilde{u}_{rj}^{e^2}}{2}+\frac{a_{rj}^{av^2}}{4})\frac{\partial^3h_i\circ l}{\partial u_{ri}\partial u_{rj}^2}(0)\\&+\tilde{u}_{ri}^e\sum_{j\not=i}^N\tilde{u}_{rj}^e\frac{\partial^3h_i\circ l}{\partial u_{ri}^2\partial u_{rj}}(0)+\frac{a_{ri}^{av^2}}{8}\frac{\partial^3h_i\circ l}{\partial u_{ri}^3}(0)\\&+\left.\sum_{j\not=i}^N\sum_{k>j, k\not=i}^N\tilde{u}_{rj}^e\tilde{u}_{rk}^e\frac{\partial^3h_i\circ l}{\partial u_{ri}\partial u_{rj}\partial u_{rk}}(0)\vphantom{\tilde{u}_{ri}^e}\right] +O(\max_ia_{ri}^{av^4}).
\end{split}
\end{equation}
Furthermore, (\ref{eq22}) is substituted in (\ref{eq23}) to compute $c_j^i$ and $d_{jk}^i$. By matching first order powers of $a_{ri}^{av}$, we have
\begin{equation}
\label{eq24}
\left[\begin{array}{c}0\\\vdots\\0\end{array}\right]=\sum_{i=1}^Na^{av}_{ri}\Delta\left[\begin{array}{c}c_i^1\\\vdots\\c_i^N\end{array}\right].
\end{equation}
Forasmuch as $\Delta$ is nonsingular, so $c_j^i$ should be equal to zero for all $i, j, k\in\{1,\hdots,N\}$. Therefore, with matching second order powers of  $a_{ri}^{av}$, the following equations are computed
\begin{equation}
\label{eq25}
\left[\begin{array}{c}0\\\vdots\\0\end{array}\right]=\sum_{j=1}^N\sum_{k>j}^Na_{rj}^{av}a^{av}_{rk}\Delta\left[\begin{array}{c}d_{jk}^1\\\vdots\\d_{jk}^N\end{array}\right].
\end{equation}

\begin{equation}
\label{eq26}
\left[\begin{array}{c}0\\\vdots\\0\end{array}\right]=\sum_{j=1}^Na_{rj}^{av^2}\left(\Delta\left[\begin{array}{c}d_{jj}^1\\\vdots\\d_{jj}^N\end{array}\right]+\left[\begin{array}{c}\frac{1}{4}\frac{\partial^3 h_1ol}{\partial u_{r1}\partial u_{rj}^2}(0)\\ \vdots\\ \frac{1}{4}\frac{\partial^3 h_{j-1}ol}{\partial u_{rj-1}\partial u_{rj}^2}(0)\\ \frac{1}{8}\frac{\partial^3h_jol}{\partial u_{rj}^3}(0)\\\frac{1}{4}\frac{\partial^3 h_{j+1}ol}{\partial u_{rj}^2\partial u_{rj+1}}(0)\\ \vdots\\\frac{1}{4}\frac{\partial^3 h_Nol}{\partial u_{rj}^2\partial u_{rN}}(0) \end{array}\right]\right).
\end{equation}
As a result, $d_{jk}^i=0$ for $\forall i\not=j$ and $d_{jj}^i$ is
\begin{equation}
\label{eq27}
\left[\begin{array}{c}d_{jj}^1\\\vdots\\d_{jj}^N\end{array}\right]=-\Delta^{-1}\left[\begin{array}{c}\frac{1}{4}\frac{\partial^3 h_1ol}{\partial u_{r1}\partial u_{rj}^2}(0)\\ \vdots\\ \frac{1}{4}\frac{\partial^3 h_{j-1}ol}{\partial u_{rj-1}\partial u_{rj}^2}(0)\\ \frac{1}{8}\frac{\partial^3h_jol}{\partial u_{rj}^3}(0)\\\frac{1}{4}\frac{\partial^3 h_{j+1}ol}{\partial u_{rj}^2\partial u_{rj+1}}(0)\\ \vdots\\\frac{1}{4}\frac{\partial^3 h_Nol}{\partial u_{rj}^2\partial u_{rN}}(0) \end{array}\right].
\end{equation}
Consequently, the equilibrium becomes
\begin{equation}
\label{eq28}
\tilde{u}_{ri}^e=\sum_{j=1}^Nd_{jj}^ia_{rj}^{av^2}+O(\max_ia_{ri}^{av^3}).
\end{equation}
Since the Jacobian  $\Gamma^{av}=(\gamma_{ij})_{(N\times N)}$ of averaging system (\ref{eq19}) is as follows:
\begin{equation}
\label{eq29}
\gamma_{ij}=\delta K_i^\prime\lim_{T\to+\infty}\frac{1}{T}\int_0^T\frac{\partial h_i\circ l}{\partial u_{rj}}(\tilde{u}_r^{av}+a_r^{av}\times\eta)\eta_i(\tau)\,d\tau
\end{equation}
by Taylor polynomial approximation and substituting average system equilibrium $\tilde{u}_r^e$ we have
\begin{equation}
\label{eq30}
\gamma_{ij}=\frac{1}{2}\delta K_i^\prime a_{ri}^{av}\frac{\partial^2h_i\circ l}{\partial u_{ri} \partial u_{rj}}(0)+O(\delta\max_ia_{ri}^{av^2})
\end{equation}
According to Assumptions \ref{as3} and \ref{as4}, and the fact that $a_{ri}^{av}$ is small and positive, (\ref{eq30}) is Hurwitz. Thus, the equilibrium (\ref{eq28}) is exponentially stable for average system (\ref{eq19}). Additionally, if $\tilde{u}_{ri}^{av}$ in (\ref{eq20}) is frozen at the equilibrium (\ref{eq28}) \cite{c20}, since $\delta\omega_{Li}^\prime>0$ and following the perturbation theory \cite{c32}, for all $0<\epsilon<\bar{\epsilon}$ where $\bar{\epsilon}>0$, (\ref{eq20}) is exponentially stable at the origin. Also, since $\tilde{n}_{ri}^e$ is center manifold, and is equal to $\frac{\partial^2h_i\circ l(0)}{2\partial u^2_{i}}\tilde{u}_{ri}^{e^2}+\frac{\partial^2h_i\circ l(0)}{4\partial u^2_{i}}a^{av^2}_{ri}+O(\vert\tilde{u}^{e}_{ri}\vert^3+\vert a^{av}_{ri}\vert^3+\vert\epsilon\vert^3)$, and as $a^{av}_{ri}$ and $\tilde{u}_{ri}^e$ converge to zero, it will converge to a small neighborhood of zero.\\
Finally, corresponding to the averaging theory and Lemma \ref{lem3}, for $0<\delta<\overline{\delta}$ and $\sigma>0$ where $\overline{\delta}>0$, the equilibrium of reduced system (\ref{eq12}) is exponentially stable, which means $(\tilde{u}_{ri}, a_{ri}, \tilde{n}_{ri})$ exponentially converges to a neighborhood of  $(\tilde{u}_{ri}^e, 0, \tilde{n}_{ri}^e)$, and the proof is completed.\vspace{5pt} \end{proof}

\subsection{Singular perturbation analysis}
We propound the following result in order to investigate the complete system (\ref{eq7}) and (\ref{eq8}) with singular perturbation theory \cite{c32} in the time scale $\tau=\bar{\omega}t$.
\begin{theorem}
\label{th2}
Consider system (\ref{eq7}) and (\ref{eq8}) under the Assumptions \ref{as1}, \ref{as2}, \ref{as3}, \ref{as4} and suppose $\omega_i\not=\omega_j$, $\omega_i\not=\omega_j+\omega_k$, $2\omega_i\not=\omega_j+\omega_k$, $\omega_i\not=2\omega_j+\omega_k$, $\omega_i\not=2\omega_j$, $\omega_i\not=3\omega_j$ and $\frac{\omega_i}{\omega_j}$ is rational, for all $i, j, k\in\{1,\hdots,N\}$, there exists a neighborhood of $(x, \hat{u}_i, a_i, n_i)=(l(u^\ast), u^\ast_i, 0, h_i\circ l(u^\ast))$ for $i\in\{1,\hdots,N\}$ along with constants $0<\epsilon<\overline{\epsilon}$, $0<\bar{\omega}<\omega^\ast$, $0<\delta<\overline{\delta}$ and  $\sigma>0$ where  $\overline{\epsilon}$, $\omega^\ast$, $\overline{\delta}>0$. Then, the solution $(x, \hat{u}_i, a_i, n_i)$ will converge exponentially to that point, and $J(t)$ will converge to $h_i\circ l(u^\ast)$ exponentially.
\end{theorem}
\begin{proof}
Corresponding to Theorem \ref{th1} and \cite{c34}, there exists a unique exponentially stable periodic solution $W_{ri}^p=\left[ \begin{array}{c} \tilde{u}_{ri}^p\\a_{ri}^p\\\tilde{n}_{ri}^p\end{array}\right]$ for $i\in\{1,\hdots,N\}$ in a neighborhood of average solution $\left[ \begin{array}{c} \tilde{u}_{i}^{av}\\a_{i}^{av}\\\tilde{n}_{i}^{av}\end{array}\right]$, which if we substitute in the system (\ref{eq8}) in the form of $\frac{dW_i}{d\tau}=\delta H_i(\tau, W, x)$, we will have
\begin{equation}
\label{eq31}
\frac{dW_{ri}^p}{d\tau}=\delta H_i(L(\tau,W_r^p),W_{ri}^p)
\end{equation}
where $L(\tau,W)=l(u^\ast+\tilde{u}+a\times \eta(\tau))$. Then, with defining $\tilde{W}_i=W_i-W_{ri}(\tau)$, we have
\begin{equation}
\label{eq32}
\bar{\omega}\frac{dx}{d\tau}=\bar{F}(\tau,\tilde{W},x)
\end{equation}
\begin{equation}
\label{eq33}
\frac{d\tilde{W}_{i}}{d\tau}=\delta \bar{H}_i(\tau,\tilde{W},x)
\end{equation}
where $\tilde{W}=[\tilde{W}_1, \hdots, \tilde{W}_N]$ and 
\begin{equation}
\label{eq34}
\begin{split}
\bar{H}_i(\tau,\tilde{W},x)&=H_i(\tau,\tilde{W}_i+W_{ri}^p(\tau),x)\\&-H_i(\tau,L(\tau,W_r^p),W_{ri}^p(\tau))
\end{split}
\end{equation}
\begin{equation}
\label{eq35}
\bar{F}(\tau,\tilde{W},x)=f\left(x,\beta(x,u^\ast+\underbrace{\tilde{u}-\tilde{u}_r^p}_{\tilde{W}_1}+\tilde{u}_r^p+a\times\eta(\tau))\right).
\end{equation}
Since $x=L(\tilde{W}+W_r^p)$ is the quasi-steady state, the reduced system is as follows:
\begin{equation}
\label{eq36}
\frac{d\tilde{W}_{ri}}{d\tau}=\delta \bar{H}_i(\tau,\tilde{W}_{ri}+W_{ri}^p,L(\tilde{W}_r+W_r^p))
\end{equation}
in which $\tilde{W}_r=0$ is the equilibrium at the origin, which is exponentially stable as it has been shown in section 3.1\\
\indent Now, the boundary layer model is studied in the time scale $t=\frac{\tau}{\bar{\omega}}$:
\begin{equation}
\label{eq37}
\begin{split}
\frac{dx_b}{dt}&=\bar{F}(x_b+L(\tilde{W}_r+W_r^p),\tilde{W})\\&=f(x_b+l(u),\beta(x_b+l(u),u))
\end{split}
\end{equation}
where $u=u^\ast+\tilde{u}+a\times \eta(\tau)$ is regarded as an independent parameter from $t$. Hence, since $f(l(u),\beta(l(u)),u)\equiv 0$ and the equilibrium of (\ref{eq37}) is $x_b=0$, according to Assumption \ref{as2}, this equilibrium is exponentially stable.\\
With exponential stability of the origin in the reduced model and boundary layer model, also with the use of Tikhonov's theorem \cite{c32}, for $0<\bar{\omega}<\omega^\ast$ where $\omega^\ast>0$, we can conclude that the solution $W(\tau)$ is close to $W_r(\tau)$ with a small neighborhood, so it exponentially converges to the periodic solution $W_r^p(\tau)$ with a small neighborhood, where $W_r^p(\tau)$ is within a neighborhood of equilibrium of the average system. As a result, $(x, \hat{u}_i, a_i, n_i)$ exponentially converges to $(l(u^\ast), u^\ast_i, 0, h_i\circ l(u^\ast))$, and consequently $J_i$ will exponentially converge to its extremum value which is $h_i\circ l(u^\ast)$, and the proof is completed.\vspace{5pt} \end{proof}

\begin{remark}
\label{rem1}
The reason of having fast convergence to NE is that when we have large initial input estimation error, the amplitude of the excitation signal will be large (see eq. (\ref{eq20})) and consequently the Jacobian matrix $\Gamma^{av}$ will be large since it is proportional to the amplitude. Also, since the amplitude will shrink gradually, the oscillation will be eliminated. This is why there is no steady state oscillation.
\end{remark}
\begin{remark}
\label{rem2}
All proofs were based on the maximization of the payoff. In order to make the problem for minimization we just need to assume $k_i < 0$, and the Jacobian matrix should be kept Hurwitz, because (\ref{eq9}) is changed to $\frac{\partial^2 j_i\circ l}{\partial u_i^2}(u^\ast)>0$.
\end{remark}
\begin{remark}
\label{rem3}
According to \cite{c35}, the systems like what we have in (\ref{eq14}) has a manifold equilibrium where if the system starts at any point on the equilibrium manifold, the system stays there and will not converge to zero. Thus, $\tilde{u}_i$ may not converge to zero and instead converges to a constant, if the system reaches to the manifold equilibrium. However, for small $\epsilon$ or in other words for large initial amplitude $a_i(0)$ the exponential stability will be guaranteed. This can happen for large initial value of input estimation error. We can have this condition by choosing proper design parameter in our ES controller. As a result, the exponential stability in this paper is guaranteed, unless the design parameters were chosen to have small initial input estimation error.   
\end{remark}
\section{Numerical Example and Simulations}
\begin{figure}[h]
\begin{center}
\includegraphics[height=9.1cm]{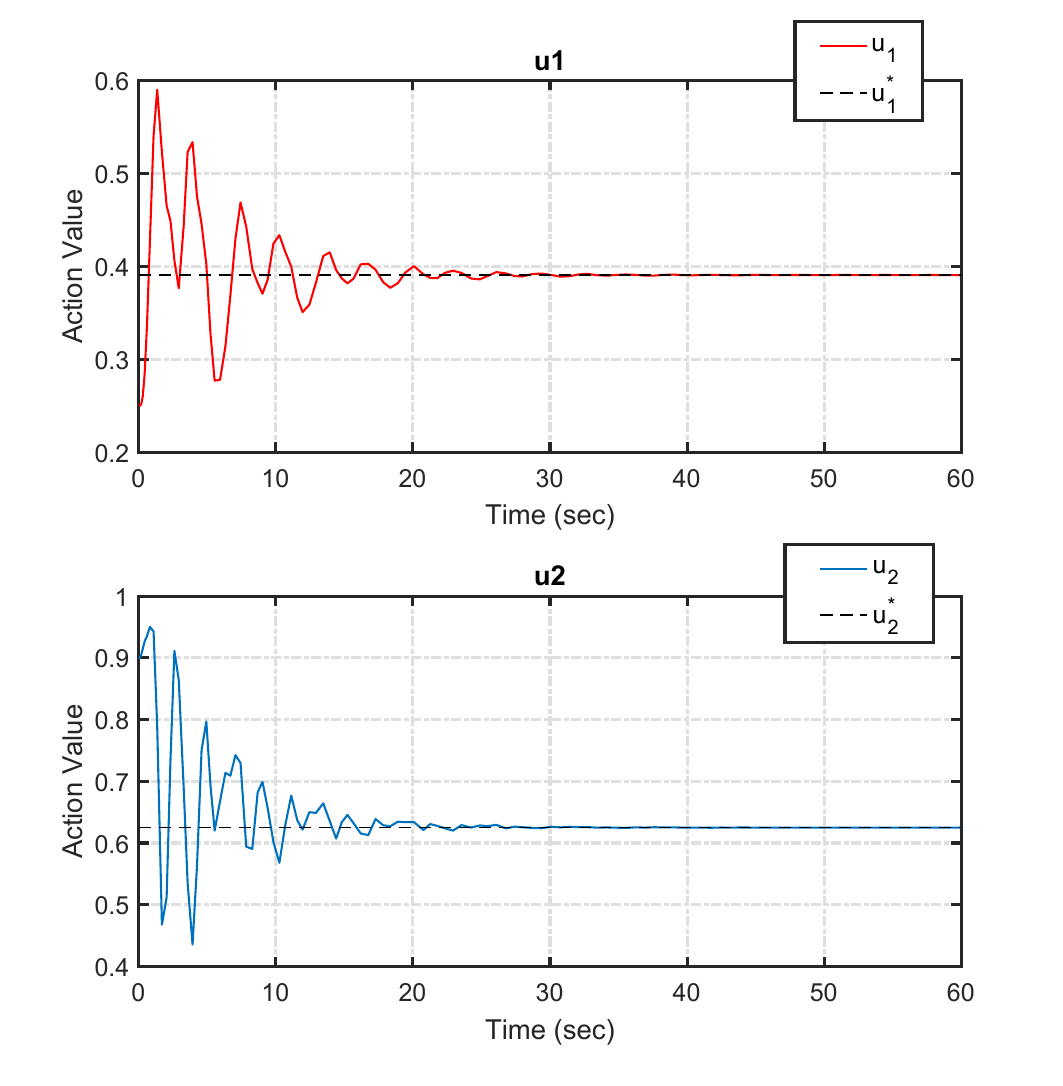}    
\caption{Action values of players as a function of time by implementing the proposed method in this paper}  
\label{fig2}                                 
\end{center}                                 
\end{figure}
\begin{figure}[h]
\begin{center}
\includegraphics[height=9.1cm]{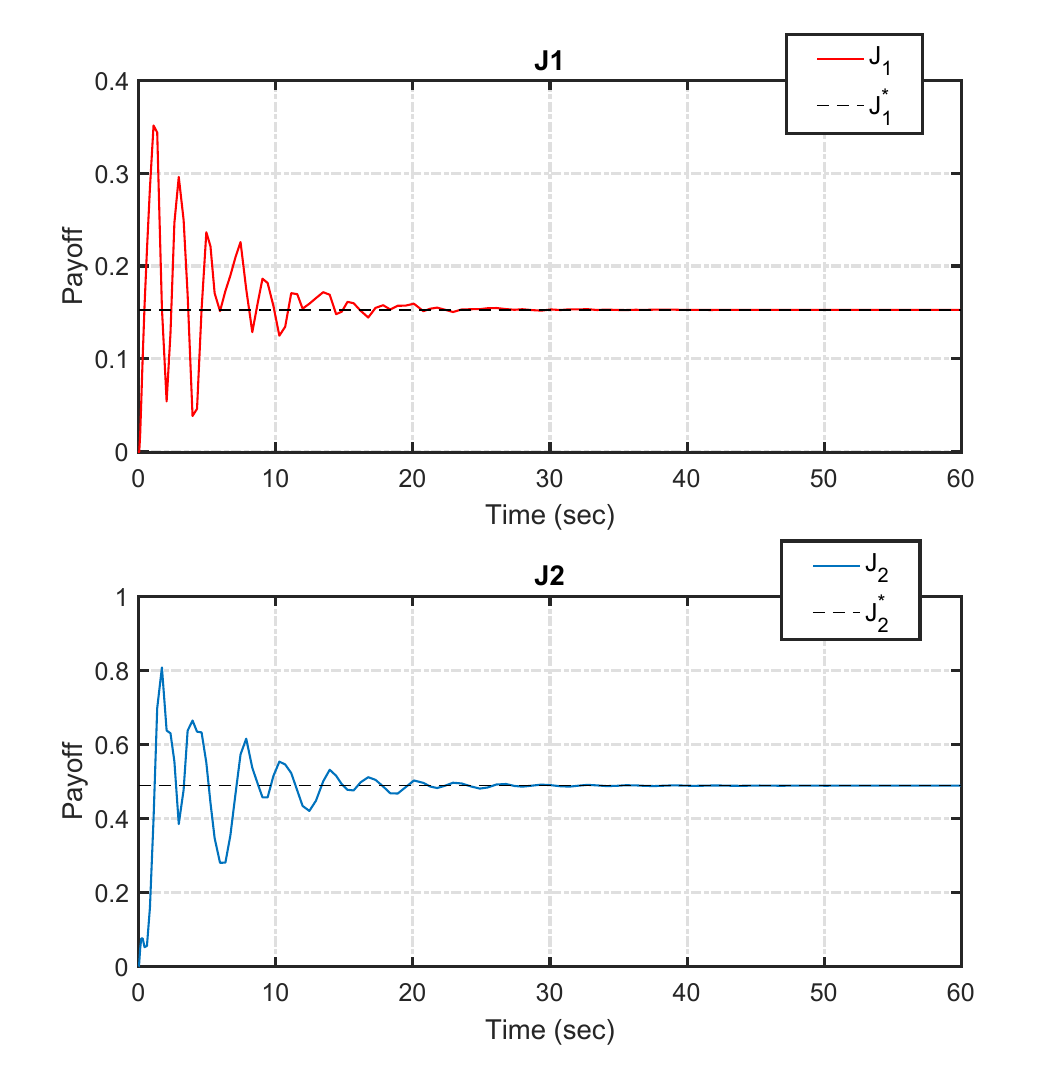}    
\caption{History of player's payoff values by employing the proposed method in this paper}  
\label{fig3}                                 
\end{center}                                 
\end{figure}
\begin{figure}[h]
\begin{center}
\includegraphics[height=9.1cm]{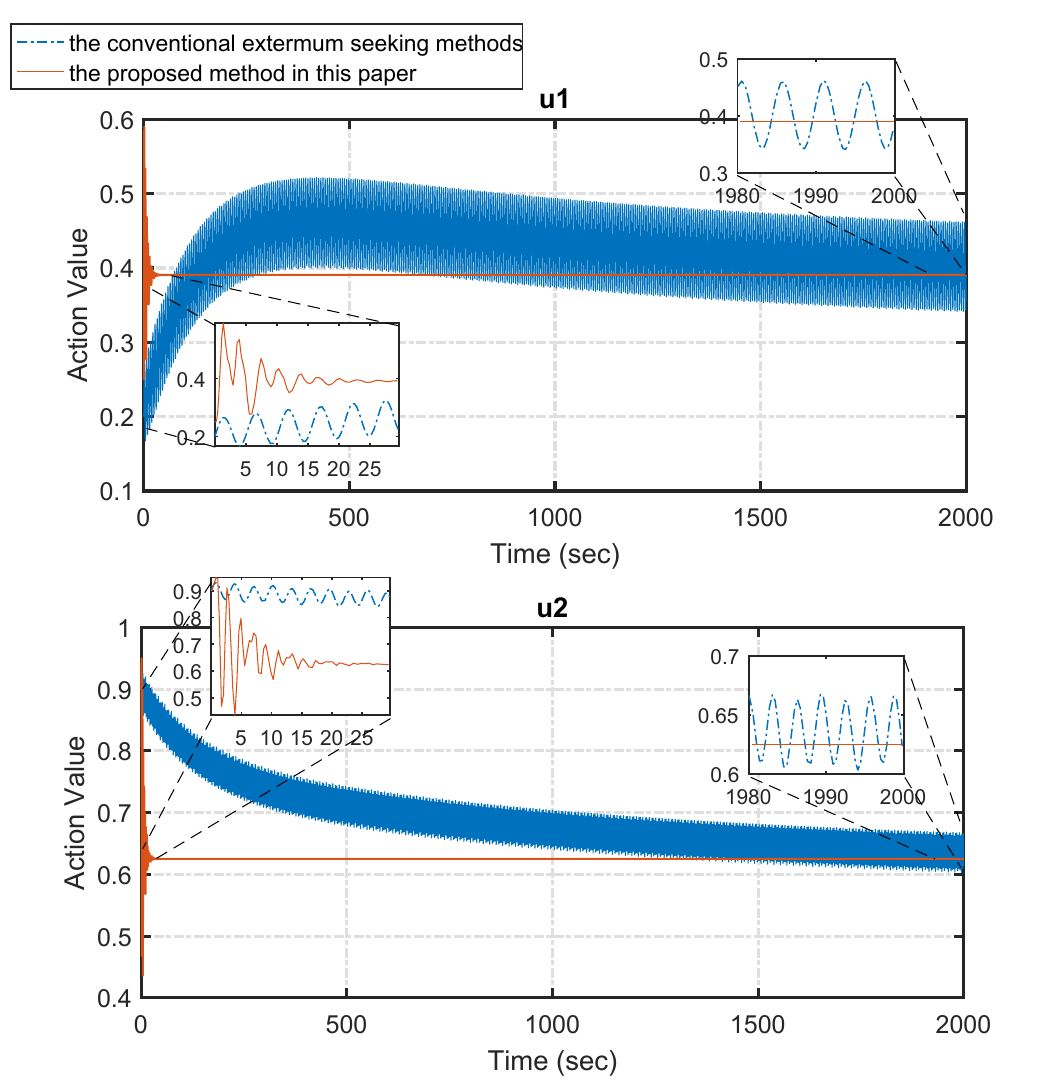}    
\caption{Action values of players as a function of time when they employ the method proposed in this paper as compared to the classical extremum seeking methods}  
\label{fig4}                                 
\end{center}                                 
\end{figure}
\label{sec4}
In this section, we consider a differential non-cooperative game so that the advantages of the proposed algorithm are investigated precisely. Oligopoly games with nonlinear demand and payoff functions \cite{c36} are real problem motivations of the kinds of games we are focusing on them. In a Cournot Oligopoly, there are different firms that produce homogeneous goods and the demands and costs are nonlinear and often dynamic and changing. Therefore, an oligopoly problem with dynamic demand and costs is one of the many motivations that we can consider for our evaluation. Consider the following differential non-cooperative game that is a form of Duopoly \cite{c23}: 
\begin{equation}
\begin{split}
\label{eq38}
\dot{x}_1&=-4x_1+x_1x_2+u_1\\
\dot{x}_2&=-4x_2+u_2
\end{split}
\end{equation}
\begin{equation}
\begin{split}
\label{eq39}
J_1&=-16x_1^2+8x_1^2x_2-x_1^2x_2^2-6x_1x_2^2+\frac{773}{32}x_1x_2-\frac{5}{8}x_1\\
J_2&=-64x_2^3+48x_1x_2-12x_1x_2^2
\end{split}
\end{equation}
The equilibrium states are calculated as
\begin{equation}
\label{eq40}
x_1^e=\frac{4u_1}{16-u_2}, \hspace{10pt} x_2^e=\frac{1}{4}u_2.
\end{equation}
Therefore, we have the following Jacobian:
\begin{equation}
\label{eq41}
\left[\begin{array}{cc} -4+\frac{1}{4}u_2& \frac{4u_1}{16-u_2}\\0&-4\end{array}\right]
\end{equation}
This matrix is Hurwitz for $u_2<16$, thus the action set of players get restricted to $\{(u_1,u_2)\in R^2 \vert u_1, u_2\geq 0, u_2<16\}$. Hence, $x^e=(x_1^e, x_2^e)$ is exponentially stable for all $(u_1, u_2)$ in the defined action set. 
Then, at $x=x^e$, the payoffs are given as follows:
\begin{equation}
\label{eq42}
J_1=-u_1^2+\frac{3}{2}u_1u_2-\frac{5}{32}u_1
\end{equation}
\begin{equation}
\label{eq43}
J_2=-u_2^3+3u_1u_2
\end{equation}
Therefore, the system has two Nash equilibria $(u_{11}^\ast, u_{21}^\ast)=(\frac{26}{64}, \frac{5}{8})$ and $(u_{12}^\ast, u_{22}^\ast)=(\frac{1}{64}, \frac{1}{8})$. $(u_{11}^\ast, u_{21}^\ast)$ admits Assumptions \ref{as3} and \ref{as4}, so it is stable, while $(u_{12}^\ast, u_{22}^\ast)$ cannot satisfy the assumptions, hence we conduct the algorithm to achieve $(u_{11}^\ast, u_{21}^\ast)$. The design parameters are selected as: $k_1=1.273$, $k_2=0.9046$, $b_1=0.7$, $b_2=0.5$, $\omega_{l1}=0.9$, $\omega_{l2}=1.5$, $\omega_{h1}=0.12$, $\omega_{h2}=0.2$, $\omega_{1}=2$, $\omega_{2}=3$ and $\phi_1=\phi_2=0$. The initial value of $(\hat{u}_1,\hat{u}_2)$ is $(\hat{u}_1(0),\hat{u}_2(0))=(0.25,0.9)$, also the state $(x_1,x_2)$ is initiated at the origin. \Cref{fig2} illustrates the history of action values for each player. \Cref{fig3} depicts the payoff values of each player as a function of time. Also, \Cref{fig4} shows a comparison between the proposed algorithm in this paper with the conventional extremum seeking control algorithms. Obviously this figure confirms the effectiveness and superiority of this algorithm with the ability of eliminating steady state oscillation and fast convergence to NE in comparison with the traditional extremum seeking control algorithms.\\
\section{Discussion}
Differential games can offer an expressive framework for a wide range of multi-agent problems that involve different agents whose decisions are impacted by others. Thus, our proposed algorithm that is offering an effective way of achieving Nash Equilibrium in differential games can be beneficial in many real-world multi-agent problems and applications. Supply chain management is one of the main streams of applications in differential games. Particularly, our algorithm is capable of solving the problem of the design of coordination management by Duopolist and Oligopolist. In such problems, we can have a multi-retailer as a differential game where a single manufacturer sells a particular product to different retailers in the same market, and each firm is trying to maximize their objective function \cite{ouardighi2013dynamic}. Smart grid with dynamic demand-side management is another broad application of our algorithm. In such problems, a differential game can be used to model the distribution of demand-side management where the price is characterized as dynamic states \cite{c11,c12}. Other areas that can directly benefit from our proposed algorithm are evasion and pursuit scenario and active defense \cite{prokopov2013linear}. Evasion and pursuit scenario can be modeled by differential games where there are conflicts between different players which are attackers, defenders, and the target. So, in such cases, our algorithm can be used for developing a feasible guidance law for an attacker \cite{qilong2018differential}. While the assumption we have might not get satisfied in all active defense problems, the problem of active defense that is occurring in the end game often follows a similar dynamic as the one proposed in this paper. Moreover, although kinematic control of a single manipulator system is well studied by different methods such as Neural-network \cite{li2021neural,tan2022discrete}, and variable damping control \cite{zahedi2021variable,zahedi2022user}, one can go beyond the direct capabilities of our proposed algorithm, and extend the problem to be applicable to kinematic control of multi-manipulator systems.   

\section{Conclusion}
\label{sec5}
In this paper, the fast extremum seeking approach introduced in \cite{c27} has been adapted to N-Player noncooperative differential games. It was shown that, to attain the Nash equilibrium (NE), each player generates its action and it was not required detailed information about payoff functions, the model and also other players' actions. Moreover, NE was achieved without oscillation and faster compared to the classical extremum seeking-based approaches. As a result, the inappropriate influence of steady-state oscillation was eliminated. Moreover, the stability and convergence analysis were presented in which the convergence without steady-state oscillation to NE was proved. Furthermore, the effectiveness of the proposed method was shown through a numerical example. Also, the comparison results of simulations between the proposed method and the conventional extremum seeking methods illustrated the superiority of the proposed algorithm.  \\

\appendix

\section{Integrals computation}
Many integrals along with Taylor polynomial approximation of $h_i\circ l$ need to be computed in order to obtain (\ref{eq23}), so for $i, j, k\in\{1, \hdots, N\}$ we have
\begin{equation}
\label{eq44}
\begin{split}
&\lim_{T\to+\infty}\frac{1}{T}\int_0^T\eta_i(\tau)\,d\tau=0, \hspace{10pt} \lim_{T\to+\infty}\frac{1}{T}\int_0^T\eta_i^2(\tau)\,d\tau=\frac{1}{2},\\
&\lim_{T\to+\infty}\frac{1}{T}\int_0^T\eta_i^3(\tau)\,d\tau=0, \hspace{10pt} \lim_{T\to+\infty}\frac{1}{T}\int_0^T\eta_i^4(\tau)\,d\tau=\frac{3}{8},
\end{split}
\end{equation}
and by making these assumptions that $\omega_i\not=\omega_j$, $2\omega_i\not=\omega_j$, and $3\omega_i\not=\omega_j$:
\begin{equation}
\label{eq45}
\begin{split}
&\lim_{T\to+\infty}\frac{1}{T}\int_0^T\eta_i(\tau)\eta_j(\tau)\,d\tau=0,\\&\lim_{T\to+\infty}\frac{1}{T}\int_0^T\eta_i^2(\tau)\eta_j(\tau)\,d\tau=0, \\
&\lim_{T\to+\infty}\frac{1}{T}\int_0^T\eta_i^3(\tau)\eta_j(\tau)\,d\tau=0,\\&\lim_{T\to+\infty}\frac{1}{T}\int_0^T\eta_i^2(\tau)\eta_j^2(\tau)\,d\tau=\frac{1}{4},
\end{split}
\end{equation}  
and these assumption that $\omega_i\not=\omega_j+\omega_k$, $\omega_i\not=2\omega_j+\omega_k$, and $2\omega_i\not=\omega_j+\omega_k$:
\begin{equation}
\label{eq46}
\begin{split}
&\lim_{T\to+\infty}\frac{1}{T}\int_0^T\eta_i(\tau)\eta_j(\tau)\eta_k(\tau)\,d\tau=0,\\
&\lim_{T\to+\infty}\frac{1}{T}\int_0^T\eta_i(\tau)\eta_j^2(\tau)\eta_k(\tau)\,d\tau=0.
\end{split}
\end{equation}
\section{Taylor polynomial approximation}
The Taylor polynomial approximation \cite{c34} of $h_i\circ l$ requires $n+1$ times differentiability of $h_i\circ l$, so the following equation is obtained:
\begin{equation}
\begin{split}
\label{eqa47}
&h_i\circ l(\tilde{u}^e+a^{av}\times\eta)=\\&\sum_{\alpha_1+\hdots+\alpha_N=0}^n\frac{\partial^{(\alpha_1+\hdots+\alpha_N)}}{\partial u_1^{\alpha_1}\hdots\partial u_N^{\alpha_N}}\frac{h_i\circ l(0)}{\alpha_1\!\hdots\alpha_N\!}(\tilde{u}^e+a^{av}\times\eta)^\alpha\\&+\sum_{\alpha_1+\hdots+\alpha_N=n+1}\frac{\partial^{(\alpha_1+\hdots+\alpha_N)}}{\partial u_1^{\alpha_1}\hdots\partial u_N^{\alpha_N}}\frac{h_i\circ l(\iota)}{\alpha_1\!\hdots\alpha_N\!}(\tilde{u}^e+a^{av}\times\eta)^\alpha\\&=\sum_{\alpha_1+\hdots+\alpha_N=0}^n\frac{h_i\circ l(0)}{\alpha_1\!\hdots\alpha_N\!}(\tilde{u}^e+a^{av}\times\eta)^\alpha\\&+O(\max_i a_i^{av^{n+1}})
\end{split}
\end{equation}
where $\iota$ is a point on the line segment of interval $[0 \hspace{8pt} \tilde{u}^e+a^{av}\times\eta(\tau)]$, $\alpha=(\alpha_1, \hdots, \alpha_N)$ and $u^\alpha$ means $u_1^{\alpha_1}\hdots u_N^{\alpha_N}$. Additionally, $O(\max_i a_i^{av^{n+1}})$ is computed on the basis of substituting (\ref{eq22}). In the process of computing (\ref{eq23}), we select $n=3$ to derive the third order derivation effect on the system as in \cite{c23}.



 \bibliographystyle{elsarticle-num} 
 \bibliography{main}





\end{document}